\documentclass[11pt]{article}
\usepackage{amsmath}
\usepackage{amsthm}
\usepackage{amssymb}
\usepackage{amscd}
\usepackage{relsize}
\usepackage{enumerate}

\newtheorem{teo}{Theorem}

\newtheorem{cor}[teo]{Corollary}
\newtheorem{lema}[teo]{Lemma}
\newcommand{\espan}{\operatorname{span}}
\DeclareMathOperator*{\esup}{ess\,sup}
\DeclareMathOperator*{\einf}{ess\,inf}

\title{{\bf Sampling in $\Lambda$-shift-invariant subspaces of Hilbert-Schmidt operators on $L^2(\mathbb{R}^d)$}}
\author{
{\bf Antonio G. Garc\'{\i}a}\thanks{E-mail:\texttt{agarcia@math.uc3m.es}}
}
\date{}
\begin{document}
\maketitle
\begin{itemize}
\item[] Departamento de Matem\'aticas, Universidad Carlos III de Madrid, Spain.
\end{itemize}
\begin{abstract}
The translation of an operator is defined by using conjugation with time-frequency shifts. Thus, one can define $\Lambda$-shift-invariant subspaces of Hilbert-Schmidt operators, finitely generated, with respect to a lattice $\Lambda$ in $\mathbb{R}^{2d}$. These spaces can be seen as a generalization of classical shift-invariant subspaces of square integrable functions. Obtaining sampling results for these subspaces appears as a natural question that can be motivated by the problem of channel estimation in wireless communications. These sampling results are obtained in the light of the frame theory in a separable Hilbert space.
\end{abstract}

\noindent{\bf AMS}: 42C15; 43A32; 47B10; 94A20.
\section{Introduction}
\label{section1}
In this paper we obtain sampling results in shift-invariant-like subspaces of the class $\mathcal{H}\mathcal{S}(\mathbb{R}^d)$ of Hilbert-Schmidt operators on $L^2(\mathbb{R}^d)$. To be more precise, these subspaces are obtained by translation in a lattice $\Lambda\subset \mathbb{R}^{2d}$ of a fixed  set of {\em Hilbert-Schmidt} operators $S_1, S_2, \dots, S_N$. {\em The translation of an operator} $S$ by $z\in \mathbb{R}^{2d}$ is defined by using conjugation with the {\em time-frequency shift} $\pi(z)$, where $z=(x, \omega)$ belongs to the {\em phase space} $\mathbb{R}^d\times \widehat{\mathbb{R}}^d$ (which in the sequel will be identified with $\mathbb{R}^{2d}$) by  
\[
\alpha_z(S):=\pi (z) S\, \pi(z)^*\,,\quad  z\in \mathbb{R}^{2d}\,. 
\]
Recall that the time-frequency shift acts on $f\in L^2(\mathbb{R}^d)$ as $\pi(z)f(t)={\rm e}^{2\pi i \omega \cdot t} f(t-x)$. The set of translations $\{\alpha_z\}_{z\in \mathbb{R}^{2d}}$ is a {\em unitary representation} of the group $\mathbb{R}^{2d}$ on the Hilbert space $\mathcal{H}\mathcal{S}(\mathbb{R}^d)$.

\medskip

If we take a {\em full rank  lattice} $\Lambda$ in $\mathbb{R}^{2d}$, i.e., $\Lambda=A\mathbb{Z}^d$ where $A$ is a $2d\times 2d$ real invertible matrix, such that the sequence $\{\alpha_\lambda (S_n)\}_{\lambda \in \Lambda;\,n=1,2,\dots,N}$ is a {\em Riesz sequence} in $\mathcal{H}\mathcal{S}(\mathbb{R}^d)$ we consider the subspace of $\mathcal{H}\mathcal{S}(\mathbb{R}^d)$ given by
\[
V_{\bf S}^2=\Big\{\sum_{n=1}^N\sum_{\lambda \in \Lambda} c_n(\lambda)\, \alpha_\lambda(S_n)\,\, :\,\, \{c_n(\lambda)\}_{\lambda \in \Lambda}\in \ell^2(\Lambda)\,,\, n=1, 2, \dots, N \Big\}\,. 
\]
From now on, the subspaces $V_{\bf S}^2$ obtained in this way will be called {\em $\Lambda$-shift-invariant subspaces} in $\mathcal{H}\mathcal{S}(\mathbb{R}^d)$. These spaces are a generalization of the classical shift-invariant subspaces in $L^2(\mathbb{R}^d)$:
\[
V^2_\Phi:=\Big\{\sum_{n=1}^N\sum_{\alpha \in \mathbb{Z}^d} c_n(\alpha)\, \varphi_n(t-\alpha)\,\, :\,\, \{c_n(\alpha)\}_{\alpha \in \mathbb{Z}^d}\in \ell^2(\mathbb{Z}^d)\,,\, n=1, 2, \dots, N \Big\}\,,
\]
where $\Phi=\{\varphi_1, \varphi_2, \dots, \varphi_{_N}\}$ denotes a set of generators of $V^2_\Phi$. Sampling in the shift-invariant subspace $V^2_\Phi$ usually involves, for each $f\in V^2_\Phi$, pointwise samples $\{f(\alpha+\beta_m)\}_{\alpha \in \mathbb{Z}^d}$ and/or average samples $\{\langle f, \psi_m(\cdot-\alpha)\rangle\}_{\alpha \in \mathbb{Z}^d}$, where $\psi_m$ is an {\em average function} in $L^2(\mathbb{R}^d)$, which not necessarily belong to $V_\varphi^2$. Any stable sampling in $V^2_\Phi$ will involve, necessarily, $M\geq N$ sequences of samples (see, for instance, \cite{aldroubi:05,garcia:06} and references therein). 

\medskip

A challenge problem here is to choose an appropriate set of samples  that should be used for operators  in $V_{\bf S}^2$. Inspired in Ref.~\cite{grochenig:14} and motivated by the problem of channel estimation in wireless communications, in this paper we propose for any $T\in V_{\bf S}^2$ its {\em diagonal channel samples} at the lattice 
$\Lambda \subset \mathbb{R}^{2d}$ defined by 
\begin{equation}
\label{samples1}
s_{_{T,m}}(\lambda):=\big\langle \alpha_{_{-\lambda}}(T)g_m, \widetilde{g}_m\big \rangle_{L^2(\mathbb{R}^d)}\,,\quad \lambda \in \Lambda\,,\,\,\, m=1, 2, \dots,M\,,
\end{equation}
where $g_m, \widetilde{g}_m$, $m=1, 2, \dots,M$, are $2M$  fixed functions in $L^2(\mathbb{R}^d)$ (we will see that necessarily $M\geq N$). The name {\em diagonal channel samples} coined for these samples will become clear later on where a little explanation will be done for both, the choice of Hilbert-Schmidt operators (in $V_{\bf S}^2$) to be sampled, and the choice of the above samples for any $T\in V_{\bf S}^2$. As we will see in Section \ref{section3.3} the samples defined in \eqref{samples1} are nothing but the {\em lower symbol of the operator $T$} with respect $g_m, \widetilde{g}_m\in L^2(\mathbb{R}^d)$ and lattice $\Lambda$, i.e., $\big\langle T\pi(\lambda)g_m, \pi(\lambda)\widetilde{g}_m \big\rangle_{L^2(\mathbb{R}^d)}$, $\lambda \in \Lambda$, or the samples of the {\em Berezin transform} $\mathcal{B}^{g_m, \widetilde{g}_m} \,T(z):=\big\langle T\pi(z)g_m, \pi(z)\widetilde{g}_m \big\rangle_{L^2(\mathbb{R}^d)}$, $z\in \mathbb{R}^{2d}$, at the lattice $\Lambda$ (see Ref.~\cite{luef:18}). These samples are also a particular case of the {\em average samples} $\langle T, \alpha_\lambda(Q_m)\big \rangle_{\mathcal{H}\mathcal{S}}$ where the {\em average operator} $Q_m$ is the rank-one operator $\widetilde{g}_m\otimes g_m$; average sampling has been used previously in Refs.~\cite{feichtinger:02,garcia:21}.

\medskip

The main aim here is the stable recovery of any $T\in V_{\bf S}^2$ from its samples \eqref{samples1} by means of a sampling formula in $V_{\bf S}^2$ having the form
\[
T=\sum_{m=1}^M \sum_{\lambda \in \Lambda}s_{_{T,m}}(\lambda)\, \alpha_\lambda(H_m) \quad \text{in $\mathcal{H}\mathcal{S}$-norm}\,,
\]
for each $T\in V_{\bf S}^2$. The operators $H_m$, $m=1,2, \dots,M$, above belong to $V_{\bf S}^2$ and satisfy that the sequence $\{\alpha_\lambda (H_m)\}_{\lambda \in \Lambda;\, m=1,2, \dots,M}$ is a {\em frame} for the Hilbert space $V_{\bf S}^2$.

\medskip

For sampling in classical shift-invariant spaces see, for instance, Refs.~\cite{aldroubi:05,garcia:06,garcia:08} and references therein. See also Ref.~\cite{hector:14} for the case where other unitary representation of $\mathbb{R}$ on $L^2(\mathbb{R})$  is used instead of the classical one given by translations. For the less known topic on sampling operators, see Refs.~\cite{feichtinger:02,garcia:21,grochenig:14,krahmer:14,pfander:13,pfander:16}.

\medskip

The used techniques in this work are those of the frame theory in a separable Hilbert space. To be precise, the  samples used along this paper will be expressed as a discrete convolution system  in the product Hilbert space $\ell^2_{_N}(\Lambda):=\ell^2(\Lambda)\times \dots \times \ell^2(\Lambda)$ ($N$ times), and then it will be used the close relationship between a discrete convolution system and a sequence of translates in $\ell^2_{_N}(\Lambda)$ (see, for instance, Ref.~\cite{garcia:20}). The other involved tools are the Kohn-Nirenberg transform or the Weyl transform for Hilbert-Schmidt operators: both are unitary operators from $L^2(\mathbb{R}^{2d})$ onto $\mathcal{H}\mathcal{S}(\mathbb{R}^d)$ which respect the translations in the sense that, if we denote any of them by $\mathcal{L}$, we have $\mathcal{L}(T_zf)=\alpha_z(\mathcal{L}f)$ for $f\in L^2(\mathbb{R}^{2d})$ and $z\in \mathbb{R}^{2d}$.

\medskip

Now we briefly explain a practical motivation for considering  the samples defined in Eq.~\eqref{samples1} for the elements in $V_{\bf S}^2$. It is a well-known fact in mobile wireless channels that the relative location between transmitter and receiver is varying with time and consequently the input-output relation is modeled by a {\em time-varying system} $x\mapsto Hx$ that can be expressed as the integral operator 
\[
Hx(t)=\int_{\mathbb{R}^d} h_t(s)\, x(t-s)\, ds=\int_{\mathbb{R}^d} \sigma (t, \omega)\, \widehat{x}(\omega)\,{\rm e}^{2\pi i \,\omega \cdot t} d\omega\,,
\]
where $\sigma(t, \omega)=\mathcal{F}(h_t)(\omega)$, i.e., the Fourier transform with respect to the last $d$ variables in $h(t,s):=h_t(s)$. In this last formulation, operator $H$ becomes a {\em pseudodifferential operator} with {\em Kohn-Nirenberg symbol} $\sigma$ (see, for instance, Refs.~\cite{grochenig:01,strohmer:06}). 

\medskip

As it was pointed out in Ref.~\cite{grochenig:14}, in {\em orthogonal frequency-division multiplexing} (OFDM) the digital information, i.e., a sequence of numbers
$\{c_\lambda\}$, $\lambda$ in the lattice $\Lambda=a\mathbb{Z}^d\times b\mathbb{Z}^d$ $(a, b>0)$, is used as the coefficients of the input signal $x(t)=\sum_{\mu \in \Lambda} c_\mu \,\pi(\lambda)g(t)$ of a time-varying system $H$ producing the output $y(t)=Hx(t)$. Then, the sequence of numbers  
\begin{equation}
\label{channel}
d_\lambda=\big\langle y,  \pi(\lambda)\widetilde{g}\, \big\rangle_{L^2(\mathbb{R}^d)} = \sum_{\mu \in \Lambda} c_\mu \,\big\langle H\pi(\mu)g, \pi(\lambda)\widetilde{g}\, \big\rangle_{L^2(\mathbb{R}^d)}\,,\quad \lambda \in \Lambda\,,
\end{equation}
is considered. The main task of the engineer is to recover the original data $\{c_\lambda\}$ from the received data $\{d_\lambda\}$. The matrix $A=[a_{\lambda, \mu}]$, where $a_{\lambda, \mu}=\big\langle H\pi(\mu)g, \pi(\lambda)\widetilde{g}\, \big\rangle_{L^2(\mathbb{R}^d)}$, which appears in Eq.~\eqref{channel}, involving $H$ and the time-frequency shifts of a pair of fixed functions $g, \widetilde{g}\in L^2(\mathbb{R}^d)$, is the so-called {\em channel matrix} associated with $H$ and the functions $g, \widetilde{g}$ in $L^2(\mathbb{R}^d)$. As it will be proved in Section \ref{section3.3} (see Eq.~\eqref{eledos} below), we have that
\[
\big\langle H\pi(\lambda)g, \pi(\lambda)\widetilde{g}\, \big\rangle_{L^2(\mathbb{R}^d)}=\big\langle \alpha_{-\lambda}(H)g, \widetilde{g}\,\big \rangle_{L^2(\mathbb{R}^d)}\,, \quad \lambda \in \Lambda\,,
\]
i.e., the samples $\big\langle \alpha_{-\lambda}(H)g, \widetilde{g}\,\big \rangle_{L^2(\mathbb{R}^d)}$, $\lambda \in \Lambda$,  coincide with the diagonal entries of the channel matrix associated with $H$ and windows $g, \widetilde{g}$. This is the reason to consider the samples defined in Eq.~\eqref{samples1} and to name  them as the {\em diagonal channel samples} of the operator $H$ with respect to the fixed functions $g, \widetilde{g}\in L^2(\mathbb{R}^d)$ and lattice $\Lambda$.

\medskip

Besides, a simple class of operators $H$ describing time-varying systems, and allowing to live in the Hilbert space setting, is given by the class of Hilbert-Schmidt operators on $L^2(\mathbb{R}^d)$. A Hilbert-Schmidt operator $H$ on $L^2(\mathbb{R}^d)$ is a compact operator on $L^2(\mathbb{R}^d)$ having the integral representation 
\[
Hx(t)=\int_{\mathbb{R}^d} \kappa(t,s)\, x(s) ds=\int_{\mathbb{R}^d} \kappa(t,t-s)\, x(t-s) ds\,,
\]
with kernel $\kappa\in L^2(\mathbb{R}^{2d})$.  Although only Hilbert-Schmidt operators on $L^2(\mathbb{R}^d)$ can be described as  integral operators with kernel in $L^2(\mathbb{R}^{2d})$, every bounded operator on $L^2(\mathbb{R}^d)$ can be uniquely described, via the {\em Schwartz kernel theorem}, by a distributional kernel in $\mathcal{S}'(\mathbb{R}^{2d})$ (see, for instance, Ref.~\cite{grochenig:01}).

\medskip

The paper is organized as follows: Section \ref{section2} introduces, for the sake of completeness, some preliminaries needed in the sequel; they  comprise Hilbert-Schmidt operators and their Kohn-Nirenberg and Weyl transforms, the concept of translation of an operator, and {\em symplectic Fourier series}. For the theory of bases and frames in a Hilbert space we cite Ref.~\cite{ole:16}. Section \ref{section3} contains the main sampling results for the multiple generated subspace $V_{\bf S}^2$ of $\mathcal{H}\mathcal{S}(\mathbb{R}^d)$. They rely on the expression of the involved samples as the output of a bounded discrete convolution system $\ell^2_{_N}(\Lambda) \rightarrow \ell^2_{_M}(\Lambda)$, and its relationship with a frame of translates for $\ell^2_{_N}(\Lambda)$.
\section{Some preliminaries}
\label{section2}
Next we briefly introduce some mathematical tools used throughout the work. For the needed theory of bases and frames in a Hilbert space we merely make reference to \cite{ole:16}; it mainly comprises Riesz sequences, dual Riesz bases and frames and its duals in a separable Hilbert space. The results for discrete convolution systems and their relationship with frames of translates in $\ell^2_{_N}(\Lambda)$ can be found, for instance, in Ref.~\cite{garcia:20}.
\subsubsection*{The Kohn-Nirenberg and Weyl transforms in the class of Hilbert-Schmidt operators}
The class of Hilbert-Schmidt operators in a Hilbert space, $L^2(\mathbb{R}^d)$ in our case, can be introduced by using the {\em Schmidt decomposition} (singular value decomposition) of a compact operator on $L^2(\mathbb{R}^d)$ (see, for instance, Ref.~\cite{conway:00}). Namely,  for a compact operator $S$ on $L^2(\mathbb{R}^d)$ there exist two orthonormal sequences $\{x_n\}_{n\in \mathbb{N}}$ and $\{y_n\}_{n\in \mathbb{N}}$ in $L^2(\mathbb{R}^d)$ and a bounded sequence of positive numbers $\{s_n(S)\}_{n\in \mathbb{N}}$ ({\em singular values} of $S$) such that 
\[
S=\sum_{n\in \mathbb{N}} s_n(S)\, x_n \otimes y_n\,,
\]
with convergence of the series in the operator norm. Here,  $x_n \otimes y_n$ denotes the rank-one operator defined by $\big(x_n \otimes y_n\big)(e)=\big\langle e, y_n\big\rangle_{L^2} x_n$ for $e\in L^2(\mathbb{R}^d)$. For $1\le p<\infty$  we define the {\em Schatten-$p$ class} $\mathcal{T}^p$ by
\[
\mathcal{T}^p:=\big\{ S \,\text{ compact on $L^2(\mathbb{R}^d)$}\, \,:\,\, \{s_n(S)\}_{n\in \mathbb{N}} \in \ell^p(\mathbb{N})\big\}\,.
\]
The Schatten-$p$ class $\mathcal{T}^p$ is a Banach space endowed with the norm $\|S\|_{\mathcal{T}^p}^p=\sum_{n\in \mathbb{N}} s_n^p(S)$.

\medskip 

In particular, for $p=1$ we obtain the so-called {\em trace class operators $\mathcal{T}^1$}. The {\em trace} defined by ${\rm tr}(S)=\sum_{n\in \mathbb{N}} \langle Se_n, e_n\rangle_{L^2}$ is a well-defined bounded linear functional on $\mathcal{T}^1$, and independent of the used orthonormal basis $\{e_n\}_{n\in \mathbb{N}}$ in $L^2(\mathbb{R}^d)$. 

\medskip

For $p=2$ we obtain the class of {\em Hilbert-Schmidt operators} $\mathcal{H}\mathcal{S}(\mathbb{R}^d):=\mathcal{T}^2$. The space $\mathcal{H}\mathcal{S}(\mathbb{R}^d)$ endowed with the inner product $\big\langle S, T \big\rangle_{\mathcal{H}\mathcal{S}}={\rm tr}(ST^*)$ becomes a Hilbert space. For the norm of $S\in \mathcal{H}\mathcal{S}(\mathbb{R}^d)$ we have
\[
\|S\|_{\mathcal{H}\mathcal{S}}^2={\rm tr}(SS^*)=\sum_{n\in \mathbb{N}}\|S^*(e_n)\|^2_{L^2}=\sum_{n\in \mathbb{N}}\|S(e_n)\|^2_{L^2}=\sum_{n\in \mathbb{N}} s_n^2(S)\,.
\]
A Hilbert-Schmidt operator $S\in \mathcal{H}\mathcal{S}(\mathbb{R}^d)$ can be seen also as a compact operator on $L^2(\mathbb{R}^d)$ defined for each $f\in L^2(\mathbb{R}^d)$ by
\[
Sf(t)=\int_{\mathbb{R}^d} \kappa_{_S}(t,x) f(x) dx\quad \text{ a.e. $t\in \mathbb{R}^d$}\,,
\]
with kernel $\kappa_{_S} \in L^2(\mathbb{R}^{2d})$. Besides, $\big\langle S, T \big\rangle_{\mathcal{H}\mathcal{S}}=\big\langle \kappa_{_S}, \kappa_{_T}\big\rangle_{L^2(\mathbb{R}^{2d})}$\, for $S, T\in \mathcal{H}\mathcal{S}(\mathbb{R}^d)$.

\medskip

Now, we briefly introduce the Kohn-Nirenberg and Weyl transforms in $L^2(\mathbb{R}^{2d})$, the setting where they will be used in this paper. More information and details about these transforms, also valid in more general settings, can be found in Refs.~\cite{feichtinger:07,folland:89,grochenig:01,skret:20,werner:84}.

\medskip

The {\em Kohn-Nirenberg transform} $L^2(\mathbb{R}^{2d}) \ni \sigma \longmapsto K_\sigma \in \mathcal{H}\mathcal{S}(\mathbb{R}^d)$ is a unitary operator where 
$K_\sigma: L^2(\mathbb{R}^d) \rightarrow L^2(\mathbb{R}^d)$ is the Hilbert-Schmidt operator defined in weak sense by
\begin{equation}
\label{KN}
\big\langle K_\sigma \phi, \psi \big\rangle_{L^2(\mathbb{R}^d)}=\big\langle \sigma, R(\psi, \phi)\big\rangle_{L^2(\mathbb{R}^{2d})}\,, \quad \phi, \psi \in L^2(\mathbb{R}^d)\,;
\end{equation}
here
\[
R(\psi, \phi)(x,\omega )=\psi(x)\, \overline{\widehat{\phi}(\omega)}\,{\rm e}^{-2\pi i \,{x \cdot \omega}}\,,\quad (x,\omega ) \in \mathbb{R}^{2d}\,,
\]
is the {\em Rihaczek distribution} of the functions $\psi, \phi \in L^2(\mathbb{R}^d)$ (see \cite[Theorem 14.6.1]{grochenig:01}).

\medskip

Thus, for each operator $S\in \mathcal{H}\mathcal{S}(\mathbb{R}^d)$ there exists a unique function $\sigma_{_S}\in L^2(\mathbb{R}^{2d})$, called its {\em Kohn-Nirenberg symbol}, i.e. $S=K_{\sigma_{_S}}$, and  such that 
\[
\langle S, T\rangle_{\mathcal{H}\mathcal{S}}=\langle \sigma_{_S}, \sigma_{_T} \rangle_{L^2(\mathbb{R}^{2d})}\quad \text{for each $S, T \in \mathcal{H}\mathcal{S}(\mathbb{R}^d)$}\,.
\]

{\em The Weyl transform} $L^2(\mathbb{R}^{2d}) \ni f\longmapsto L_f \in \mathcal{H}\mathcal{S}(\mathbb{R}^d)$ is also a unitary operator where 
$L_f: L^2(\mathbb{R}^d) \rightarrow L^2(\mathbb{R}^d)$ is the Hilbert-Schmidt operator defined in weak sense by
\begin{equation}
\label{W}
\big\langle L_f \phi, \psi \big\rangle_{L^2(\mathbb{R}^d)}=\big\langle f, W(\psi, \phi)\big\rangle_{L^2(\mathbb{R}^{2d})}\,, \quad \phi, \psi \in L^2(\mathbb{R}^d)\,;
\end{equation}
here
\[
W(\psi, \phi)(x,\omega )=\int_{\mathbb{R}^d}\psi\big(x+\frac{t}{2}\big)\, \overline{\phi\big(x-\frac{t}{2}\big)}\,{\rm e}^{-2\pi i \,{\omega \cdot t}}dt\,,\quad (x,\omega ) \in \mathbb{R}^{2d}\,,
\]
is the {\em cross-Wigner distribution} of the functions $\psi, \phi \in L^2(\mathbb{R}^d)$ (see Ref.~\cite[Theorem 14.6.1]{grochenig:01}).

\medskip

Thus, for each operator $S\in \mathcal{H}\mathcal{S}(\mathbb{R}^d)$ there exists a unique function $a_{_S}\in L^2(\mathbb{R}^{2d})$, called its {\em Weyl symbol}, i.e. $S=L_{a_{_S}}$, and  such that 
\[
\langle S, T\rangle_{\mathcal{H}\mathcal{S}}=\langle a_{_S}, a_{_T} \rangle_{L^2(\mathbb{R}^{2d})}\quad \text{for each $S, T \in \mathcal{H}\mathcal{S}(\mathbb{R}^d)$}\,.
\]
If $a_{_S}$ denotes the Weyl symbol  of $S$, its Kohn-Nirenberg symbol $\sigma_{_S}$ is given by $Ua_{_S}$ where $U$ is the unitary operator on $L^2(\mathbb{R}^{2d})$ such that $\widehat{Ua_{_S}}(\xi, u)={\rm e}^{\pi i u\cdot\xi}\, \widehat{a}_{_S}(\xi, u)$, $(\xi, u)\in \mathbb{R}^{2d}$ (see the details in Ref.~\cite{grochenig:01}).

The Kohn-Nirenberg (or Weyl) transform can be defined for $\sigma$ (or $f$) in $\mathcal{S}'(\mathbb{R}^{2d})$, i.e., for tempered distributions by using the dualities 
$\big(\mathcal{S}(\mathbb{R}^d), \mathcal{S}'(\mathbb{R}^d)\big)$ and $\big(\mathcal{S}(\mathbb{R}^{2d}), \mathcal{S}'(\mathbb{R}^{2d})\big)$ in Eq.~\eqref{KN} (or Eq.~\eqref{W}); see, for instance, Refs.~\cite{grochenig:01,skret:20}.
\subsubsection*{Translation of operators}
For $z=(x, \omega) \in \mathbb{R}^{2d}$, the {\em time-frequency shift} operator $\pi(z): L^2(\mathbb{R}^d) \rightarrow L^2(\mathbb{R}^d)$ is defined as 
\[
\pi(z) \varphi(t)={\rm e}^{2\pi i \omega\cdot t}\varphi(t-x)\quad \text{for $\varphi \in L^2(\mathbb{R}^d)$}\,.
\]
It is used to define the {\em short-time Fourier transform} (Gabor transform) $V_\psi\varphi$ of $\varphi$ with window $\psi$, both in $L^2(\mathbb{R}^d)$, by
\[
V_\psi\varphi(z)=\big\langle \varphi, \pi(z)\psi \big\rangle_{L^2(\mathbb{R}^d)}\,,\quad z \in \mathbb{R}^{2d}\,.
\]
Its adjoint operator is $\pi(z)^*={\rm e}^{-2\pi i x\cdot \omega}\, \pi(-z)$ for $z=(x, \omega)\in \mathbb{R}^{2d}$.
By using conjugation with $\pi(z)$ one can define the translation by $z\in \mathbb{R}^{2d}$ of an operator $S \in \mathcal{H}\mathcal{S}(\mathbb{R}^d)$. Namely,
\[
\alpha_z(S):=\pi(z)\,S\,\pi(z)^*\,,\quad z \in \mathbb{R}^{2d}\,.
\]
For instance, for $\varphi, \psi \in L^2(\mathbb{R}^d)$ we get $\alpha_z(\varphi \otimes \psi)=[\pi(z)\varphi] \otimes [\pi(z)\psi]$, \,$z \in \mathbb{R}^{2d}$.

Since $\alpha_z$ defines a unitary operator on $\mathcal{H}\mathcal{S}(\mathbb{R}^d)$, $\alpha_z \alpha_{z'}=\alpha_{z+z'}$ for $z, z' \in \mathbb{R}^{2d}$, and the map $z\mapsto \alpha_z(S)$ is continuous for each $S\in \mathcal{H}\mathcal{S}(\mathbb{R}^d)$ we have that  
$\big\{\alpha_z\big\}_{z\in \mathbb{R}^{2d}}$ is a {\em unitary representation} of the group $\mathbb{R}^{2d}$ on the Hilbert space 
$\mathcal{H}\mathcal{S}(\mathbb{R}^d)$. More properties and applications can be found, for instance, in Refs.~\cite{luef:18,skret:20, werner:84}.

\subsubsection*{Symplectic Fourier series}
Let $\Lambda$ be a {\em full rank lattice} in $\mathbb{R}^{2d}$, i.e., $\Lambda=A\mathbb{Z}^{2d}$ with $A\in GL(2d,\mathbb{R})$ and volume $|\Lambda|=\det A$. Its dual group $\widehat{\Lambda}$ is identified with $\mathbb{R}^{2d}/\Lambda^\circ$, where $\Lambda^\circ$ is the {\em annihilator group}
\[
\Lambda^\circ=\big\{\lambda^\circ \in \mathbb{R}^{2d}\, \, :\, \, {\rm e}^{2\pi i\, \sigma(\lambda^\circ , \lambda)}=1\,\, \text{ for all $\lambda \in \Lambda$} \big\}\,,
\]
where $\sigma$ denotes here the {\em standard symplectic form} $\sigma (z,z')=\omega\cdot x'-\omega'\cdot x$ \, for $z=(x, \omega)$ and $z'=(x',\omega')$ in $\mathbb{R}^{2d}$. Notice that, since $\Lambda$ is discrete its dual group $\widehat{\Lambda}$ is compact. The group $\Lambda^\circ$ is itself a lattice: the so-called {\em adjoint lattice} of $\Lambda$. The {\em symplectic characters} $\chi_z(z'):={\rm e}^{2\pi i \,\sigma(z,z')}$ are the natural way of identifying the group 
$\mathbb{R}^{2d}$ with its dual group via the bijection $z\mapsto \chi_z$. 

\medskip

The  Fourier transform of $c\in \ell^1(\Lambda)$ is the {\em symplectic Fourier series}
\[
\mathcal{F}_s^\Lambda(c)(\dot z):=\sum_{\lambda \in \Lambda} c(\lambda)\, {\rm e}^{2\pi i \,\sigma(\lambda, z)}\,, \quad \dot z \in \mathbb{R}^{2d}/\Lambda^\circ\,,
\]
where $\dot z$ denotes the image of $z$ under the natural quotient map $\mathbb{R}^{2d} \rightarrow \mathbb{R}^{2d}/\Lambda^\circ$.

\medskip

Since $\mathcal{F}_s^\Lambda$ is a Fourier transform it extends to a unitary mapping $\mathcal{F}_s^\Lambda : \ell^2(\Lambda) \rightarrow L^2(\,\widehat{\Lambda}\,)$. It satisfies $\mathcal{F}_s^\Lambda(c\ast_\Lambda d)=\mathcal{F}_s^\Lambda(c)\,\mathcal{F}_s^\Lambda(d)$ for $c\in \ell^1(\Lambda)$ and $d\in \ell^2(\Lambda)$. Moreover, if $c, d \in \ell^2(\Lambda)$ with $c\ast_\Lambda d \in \ell^2(\Lambda)$, then $\mathcal{F}_s^\Lambda(c\ast_\Lambda d)=\mathcal{F}_s^\Lambda(c)\,\mathcal{F}_s^\Lambda(d)$. As usual, the convolution $\ast_\Lambda$ of two sequences $c, d$  is defined by
\[
\big(c \ast_\Lambda d\big)(\lambda)=\sum_{\mu\in \Lambda} c(\mu)\, d(\lambda-\mu),\quad \lambda \in \Lambda\,.
\]
For more details, see, for instance, Refs.~\cite{deitmar:14,folland:95, fuhr:05,skret:20}.

\section{Sampling in the case of multiple generators}
\label{section3}
For a fixed set ${\bf S}=\{S_1, S_2, \dots, S_N\} \subset \mathcal{H}\mathcal{S}(\mathbb{R}^d)$, we are interested that  the sequence of translates $\{\alpha_\lambda (S_n)\}_{\lambda \in \Lambda;\, n=1,2,\dots,N}$ forms a Riesz sequence for $\mathcal{H}\mathcal{S}(\mathbb{R}^d)$ where $\Lambda \subset \mathbb{R}^{2d}$ is a full rank lattice with dual group $\widehat{\Lambda}$. 
\subsection{Riesz sequences of translated operators in $\mathcal{H}\mathcal{S}(\mathbb{R}^d)$}
As it was said before, the Weyl transform $f\mapsto L_f$ is a unitary operator $L^2(\mathbb{R}^{2d}) \rightarrow \mathcal{H}\mathcal{S}(\mathbb{R}^d)$ which respects translations in the sense that
\[
L_{T_zf}=\alpha_z(L_f)\quad \text{for $f\in L^2(\mathbb{R}^{2d})$ and $z \in \mathbb{R}^{2d}$}\,.
\]
These two properties are very important throughtout this work. In particular, as it was pointed out in Refs.~\cite{feichtinger:02,skret:20}, for fixed $S\in \mathcal{H}\mathcal{S}(\mathbb{R}^d)$ with Weyl symbol $a_{_S} \in L^2(\mathbb{R}^{2d})$ and lattice $\Lambda$ in $\mathbb{R}^{2d}$, the sequence 
$\{\alpha_\lambda (S)\}_{\lambda \in \Lambda}$ is a Riesz sequence in $\mathcal{H}\mathcal{S}(\mathbb{R}^d)$, i.e., a Riesz basis for $V_S^2:=\overline{\espan}_{\mathcal{H}\mathcal{S}}\big\{ \alpha_\lambda(S)\big\}_{\lambda \in \Lambda}$, if and only if the sequence $\{T_\lambda (a_{_S})\}_{\lambda \in \Lambda}$ is a Riesz sequence in $L^2(\mathbb{R}^{2d})$, i.e., a Riesz basis for the shift-invariant subspace $V_{a_S}^2$ in $L^2(\mathbb{R}^{2d})$ generated by $a_{_S}$. 

A necessary and sufficient condition for $\{\alpha_\lambda (S)\}_{\lambda \in \Lambda}$ to be a Riesz sequence in $\mathcal{H}\mathcal{S}(\mathbb{R}^d)$ 
is given in Ref.~\cite{skret:20}. There, it is assumed that $S\in \mathcal{B}$, a Banach space of continuous operators with Weyl symbol $a_{_S}$ in the {\em Feichtinger's algebra} $\mathcal{S}_0(\mathbb{R}^{2d})$; in essence, 
$\mathcal{B}$ consists of trace class operators on $L^2(\mathbb{R}^d)$ with a norm-continuous inclusion 
$\iota : \mathcal{B} \hookrightarrow \mathcal{T}^1$  (see the details in Refs.~\cite{grochenig:99,skret:20}). 

\medskip

Recall that the {\em Feichtinger's algebra} $\mathcal{S}_0(\mathbb{R}^d)$ is the space of all tempered distributions $\psi$ in $\mathbb{R}^d$ such that 
\[
\|\psi\|_{\mathcal{S}_0}:=\int_{\mathbb{R}^{2d}} |V_{\varphi_0}\psi(z)|dz <\infty\,,
\]
where $\varphi_0$ denotes the $L^2$-normalized gaussian $\varphi_0(x)=2^{d/4} {\rm e}^{-\pi x\cdot x}$ for $x\in \mathbb{R}^d$. With this norm, 
$\mathcal{S}_0(\mathbb{R}^d)$ is a Banach space of continuous functions and an algebra under multiplication and convolution; see the details in Refs.~\cite{grochenig:01,jakobsen:18,skret:20}.

\begin{teo}(\cite[Theorem 6.1]{skret:20})
Let $\Lambda$ be a lattice and $S\in \mathcal{B}$. The sequence $\{\alpha_\lambda (S)\}_{\lambda \in \Lambda}$ is a Riesz sequence in $\mathcal{H}\mathcal{S}(\mathbb{R}^d)$ if and only if the function
\[
P_{\Lambda^\circ}\big(|\mathcal{F}_W(S)|^2\big)(\dot z):=\frac{1}{|\Lambda|}\sum_{\lambda^\circ \in \Lambda^\circ} |\mathcal{F}_W(S)(z+\lambda^\circ)|^2\,,\quad z\in \mathbb{R}^{2d}\,,
\] 
has no zeros in $\widehat{\Lambda}$.
\end{teo}
It involves the {\em periodization operator} $P_{\Lambda^\circ}$ in $\Lambda^\circ$ and the {\em Fourier-Wigner transform} $\mathcal{F}_W$ of an operator $S$. In this case, we have that $\mathcal{F}_W(S)=\mathcal{F}_s(a_{_S})$, where $\mathcal{F}_s$ denotes the {\em symplectic Fourier transform} of $a_{_S}$ defined by
\[
\mathcal{F}_s(a_{_S})(z):=\int_{\mathbb{R}^{2d}} a_{_S}(z')\, {\rm e}^{-2\pi i \,\sigma(z, z')}dz'\,,\quad z\in \mathbb{R}^{2d}\,,
\]
where $\sigma$ denotes here the standard symplectic form in $\mathbb{R}^{2d}$. The Fourier-Wigner transform of an operator $S$ is defined as the function
\[
\mathcal{F}_W(S)(z):={\rm e}^{-\pi i \,x\cdot \omega}\, {\rm tr}[\pi(-z)S]\,, \quad z=(x,\omega)\in \mathbb{R}^{2d}\,.
\]
See the details in Ref.~\cite{skret:20}. A similar result to that in the above theorem for a rank-one operator $S=\psi\otimes \phi$, where $\psi, \phi \in  L^2(\mathbb{R}^d)$, can be found in Refs.~\cite{benedetto:06,feichtinger:02}.

\medskip

In case $\{\alpha_\lambda (S)\}_{\lambda \in \Lambda}$ is a Riesz sequence for $\mathcal{H}\mathcal{S}(\mathbb{R}^d)$, the operator $S$ is the {\em generator} of the {\em $\Lambda$-shift-invariant subspace} $V_S^2$ which can be described by 
\[
V_S^2:=\overline{\espan}_{\mathcal{H}\mathcal{S}}\big\{ \alpha_\lambda(S)\big\}_{\lambda \in \Lambda}=\Big\{\sum_{\lambda \in \Lambda} c(\lambda)\, \alpha_\lambda(S)\,\, :\,\, \{c(\lambda)\}_{\lambda \in \Lambda}\in \ell^2(\Lambda) \Big\}\,.
\]
Observe that operators in $V_S^2$ are nothing but {\em Gabor multipliers} in case $S=\varphi\otimes\psi$. Indeed, for 
$\eta \in L^2(\mathbb{R}^d)$ we have
\[
\sum_{\lambda\in \Lambda} c(\lambda)\,\alpha_\lambda(S)(\eta)=\sum_{\lambda\in \Lambda} c(\lambda)\,\big(\pi(\lambda) \varphi \otimes \pi(\lambda)\psi\big)(\eta) =\sum_{\lambda\in \Lambda} c(\lambda)\,V_\psi\,\eta (\lambda)\pi(\lambda)\varphi\,,
\]
that is, $\sum_{\lambda\in \Lambda} c(\lambda)\,\alpha_\lambda(S)=\mathcal{G}_{\mathbf{c}}^{\psi,\varphi}$, the Gabor multiplier with windows $\psi, \varphi$ and mask $\mathbf{c}$ in $\ell^2(\Lambda)$ used in time-frequency analysis (see, for instance, Ref.~\cite{skret:20}).

\medskip

Analogously, a necessary and sufficient condition can be obtained for the multiply generated  case. Indeed, let ${\bf S}=\{S_1, S_2, \dots, S_N\}$ be a fixed subset of $\mathcal{H}\mathcal{S}(\mathbb{R}^d)$ and let $\Lambda$ be a lattice in $\mathbb{R}^{2d}$. We are searching for a necessary and sufficient condition such that  $\{\alpha_\lambda (S_n)\}_{\lambda \in \Lambda;\, n=1,2,\dots,N}$ is a Riesz sequence for $\mathcal{H}\mathcal{S}(\mathbb{R}^d)$, i.e., a Riesz basis for the closed subspace
\[
V_{\bf S}^2:=\overline{\espan}_{\mathcal{H}\mathcal{S}}\big\{ \alpha_\lambda(S_n)\big\}_{\lambda \in \Lambda;\,n=1,2, \dots,N}\subset \mathcal{H}\mathcal{S}(\mathbb{R}^d)\,.
\]
For the multiply generated case we have the following result:
\begin{teo}
\label{Ngenerators}
Let $\Lambda$ be a lattice and $S_n\in \mathcal{B}$, $n=1,2,\dots,N$. Then, $\{\alpha_\lambda (S_n)\}_{\lambda \in \Lambda;\, n=1,2,\dots,N}$ is a Riesz sequence for $\mathcal{H}\mathcal{S}(\mathbb{R}^d)$ if and only if there exist two constants $0<m\le M$ such that 
\[
m\,\mathbb{I}_{_N} \le G_{\bf S}^W(z) \le M\,\mathbb{I}_{_N} \quad  \text{for any $z\in \mathbb{R}^{2d}$}\,,
\]
where $G_{\bf S}^W(z)$ denotes the $N\times N$ matrix-valued function
\[
G_{\bf S}^W(z):=\sum_{\lambda^\circ \in \Lambda^\circ} \mathcal{F}_W({\bf S})(z+\lambda^\circ)\, \overline{\mathcal{F}_W({\bf S})(z+\lambda^\circ)}^\top\,, \quad z\in \mathbb{R}^{2d}\,,
\]
and $\mathcal{F}_W({\bf S})=\big(\mathcal{F}_W(S_1), \mathcal{F}_W(S_2), \dots,\mathcal{F}_W(S_N)\big)^\top$.
\end{teo}
\begin{proof}
As indicated above, it will be a Riesz sequence in $\mathcal{H}\mathcal{S}(\mathbb{R}^d)$ if and only if the sequence $\{T_\lambda (a_{_{S_n}})\}_{\lambda \in \Lambda;\,n=1,2,\dots,N}$ is a Riesz sequence in $L^2(\mathbb{R}^{2d})$. To this end,  we introduce the $N\times N$ matrix-valued function
\[
G_{\bf S}^\sigma(z):=\sum_{\lambda^\circ \in \Lambda^\circ} \mathcal{F}_s(a_{_{\bf S}})(z+\lambda^\circ)\, \overline{ \mathcal{F}_s(a_{_{\bf S}})(z+\lambda^\circ)}^\top\,, \quad z\in \mathbb{R}^{2d}\,,
\]
where $\mathcal{F}_s(a_{_{\bf S}})=\big(\mathcal{F}_s(a_{_{S_1}}), \mathcal{F}_s(a_{_{S_2}}), \dots,\mathcal{F}_s(a_{_{S_N}})\big)^\top$. It is known (see, for instance, Ref.~\cite{aldroubi:05}) that the sequence $\{T_\lambda (a_{_{S_n}})\}_{\lambda \in \Lambda;\,n=1,2,\dots,N}$ is a Riesz sequence in $L^2(\mathbb{R}^{2d})$ if and only if there exist two constants $0<m\le M$ such that $m\,\mathbb{I}_{_N} \le G_{\bf S}^\sigma(z) \le M\,\mathbb{I}_{_N}$, a.e. $z\in \mathbb{R}^{2d}$, where $\mathbb{I}_{_N}$ denotes the $N\times N$ identity matrix. 
Assuming as before that $S_n\in \mathcal{B}$, $n=1,2,\dots,N$, the functions $\mathcal{F}_s(a_{_{S_n}})$ are continuous and $\mathcal{F}_W(S_n)=\mathcal{F}_s(a_{_{S_n}})$ for $n=1,2,\dots,N$. Hence, the above necessary and sufficient condition can be expressed in terms of the hermitian matrix $G_{\bf S}^W(z)$ as in the statement of the theorem.
\end{proof}

In this case, ${\bf S}=\{S_1, S_2, \dots, S_N\}$ is a {\em set of generators} for the {\em $\Lambda$-shift-invariant subspace} $V_{\bf S}^2:=\overline{\espan}_{\mathcal{H}\mathcal{S}}\big\{ \alpha_\lambda(S_n)\big\}_{\lambda \in \Lambda;\,n=1,2, \dots,N}$ which can be described by 
\[
V_{\bf S}^2=\Big\{\sum_{n=1}^N\sum_{\lambda \in \Lambda} c_n(\lambda)\, \alpha_\lambda(S_n)\,\, :\,\, \{c_n(\lambda)\}_{\lambda \in \Lambda}\in \ell^2(\Lambda)\,,\, n=1, 2, \dots, N \Big\}\,.
\]
\subsection{The isomorphism $ \mathcal{T}_{\bf S}$}
Our sampling results rely on the following isomorphism $\mathcal{T}_{\bf S}$ which involves the spaces $\ell^2_{_N}(\Lambda)$, the shift-invariant subspace $V_{\sigma_{_{\bf S}}}^2$ in $L^2(\mathbb{R}^{2d})$ generated by the Kohn-Nirenberg symbols $\sigma_{_{S_n}}$ of $S_n$, $n=1,2, \dots,N$,  and the $\Lambda$-shift-invariant subspace $V_{\bf S}^2$. 
Namely,
\begin{equation}
\label{isos}
\begin{array}[c]{cccccc}
 \mathcal{T}_{\bf S}:&  \ell^2_{_N}(\Lambda) & \longrightarrow & V_{\sigma_{\bf S}}^2 \subset L^2(\mathbb{R}^{2d}) & \longrightarrow & V_{\bf S}^2 \subset \mathcal{H}\mathcal{S}(\mathbb{R}^d)
 \phantom{\dfrac{a}{b}} \\
       & (c_1, c_2, \dots, c_N)^\top & \longmapsto &\displaystyle{\sum_{n=1}^N \sum_{\lambda \in \Lambda} c_n(\lambda)\, T_\lambda  \sigma_{_{S_n}}}  & \longmapsto  &\displaystyle{\sum_{n=1}^N \sum_{\lambda \in \Lambda} c_n(\lambda)\, \alpha_\lambda(S_n)}\,.
\end{array}
\end{equation}
The isomorphism $\mathcal{T}_{\bf S}$ is the composition of the  isomorphism $\mathcal{T}_{\sigma_{\bf S}}:\ell^2_{_N}(\Lambda)\rightarrow V_{\sigma_{\bf S}}^2$ which maps the standard orthonormal basis 
$\{\boldsymbol{\delta}_\lambda\}_{\lambda \in \Lambda}$ for $\ell^2_{_N}(\Lambda)$ onto the Riesz basis $\{T_\lambda  \sigma_{_{S_n}}\}_{\lambda \in \Lambda;\,n=1,2,\dots,N}$ for $V_{\sigma_{\bf S}}^2$, and the Kohn-Nirenberg transform transform between $V_{\sigma_{\bf S}}^2$ and $V_{\bf S}^2$. 

Recall that the Kohn-Nirenberg transform $L^2(\mathbb{R}^{2d}) \ni f\mapsto K_f\in  \mathcal{H}\mathcal{S}(\mathbb{R}^d) $ is a unitary operator which respects translations in the sense that $K_{T_zf}=\alpha_z(K_f)$ for $f\in L^2(\mathbb{R}^{2d})$ and $z \in \mathbb{R}^{2d}$. See, for instance, Ref.~\cite{feichtinger:07,grochenig:01}.

\subsection{An expression for the samples}
\label{section3.3}
For each $T=\sum_{n=1}^N\sum_{\mu \in \Lambda} c_n(\mu)\, \alpha_\mu(S_n)$ in $V_{\bf S}^2$ we define a set of  {\em diagonal channel samples} as
\begin{equation}
\label{dchsamp}
 \mathbf{s}_{_T}(\lambda):=\big(\langle \alpha_{-\lambda}(T) g_1, \widetilde{g}_1 \rangle, \langle \alpha_{-\lambda}(T) g_2, \widetilde{g}_2 \rangle, \dots, \langle \alpha_{-\lambda}(T) g_{_M}, \widetilde{g}_{_M} \rangle\big)^\top\,,\quad \lambda \in \Lambda\,,
\end{equation}
where $g_m, \widetilde{g}_m$, $m=1, 2, \dots, M$, denote $2M$ fixed functions in $L^2(\mathbb{R}^d)$. For $m=1, 2, \dots, M$ the above samples can be expressed by
\begin{equation}
\label{samples2}
\begin{split}
s_{_{T,m}}(\lambda):=&\big\langle \alpha_{-\lambda}(T) g_m, \widetilde{g}_m \big\rangle_{L^2(\mathbb{R}^d)}= \big\langle \sum_{n=1}^N \sum_{\mu \in \Lambda} c_n(\mu) \alpha_{\mu-\lambda}(S_n)g_m, \widetilde{g}_m \big\rangle_{L^2(\mathbb{R}^d)} \\
=&\sum_{n=1}^N \sum_{\mu \in \Lambda} c_n(\mu) \big\langle \alpha_{\mu-\lambda}(S_n)g_m, \widetilde{g}_m \big\rangle_{L^2(\mathbb{R}^d)}
=\sum_{n=1}^N \big(a_{m,n} \ast_\Lambda c_n\big)(\lambda)\,, \quad \lambda \in \Lambda\,,
\end{split}
\end{equation}
where $a_{m,n}(\mu):=\big\langle \alpha_{-\mu}(S_n) g_m, \widetilde{g}_m \big\rangle_{L^2(\mathbb{R}^d)}$, \,$\mu \in \Lambda$. Observe that $a_{m,n}(\lambda)$, $\lambda \in \Lambda$, are precisely the samples $ \mathbf{s}_{_{S_n}}(\lambda)$, $\lambda \in \Lambda$, of the generator $S_n$.
\begin{lema}
\label{lemmasamples}
Concerning the samples defined in Eq.~\eqref{samples2} we have:
\begin{enumerate}
\item For $m=1, 2, \dots, M$ these samples can be written as
\begin{equation}
\label{eledos}
\big\langle \alpha_{-\lambda}(T) g_m, \widetilde{g}_m \big\rangle_{L^2(\mathbb{R}^d)}=\big\langle T\pi(\lambda)g_m, \pi(\lambda)\widetilde{g}_m\big\rangle_{L^2(\mathbb{R}^d)}=\big\langle T, \alpha_{\lambda}(\widetilde{g}_m\otimes g_m)\big\rangle_{_{\mathcal{H}\mathcal{S}}}\,,\,\, \lambda \in \Lambda\,.
\end{equation}
\item The sequences $\big\{a_{m,n}(\lambda)\big\}_{\lambda \in \Lambda}$ appearing in Eq.~\eqref{samples2} belong to $\ell^2(\Lambda)$ for $m=1,2, \dots,M$ and $n=1,2, \dots,N$.
\end{enumerate}
\end{lema}
\begin{proof}
For the first equality in \eqref{eledos} we have that
\[
\begin{split}
s_{_{T,m}}(\lambda)&=\big\langle \alpha_{-\lambda}(T) g_m, \widetilde{g}_m \big\rangle_{L^2(\mathbb{R}^d)}=\big\langle \pi(-\lambda) T \pi(-\lambda)^*g_m, \widetilde{g}_m \big\rangle_{L^2(\mathbb{R}^d)}\\
&=\big\langle \sigma_{_T}, R\big(\pi(-\lambda)^*\widetilde{g}_m, \pi(-\lambda)^*g_m\big) \big\rangle_{L^2(\mathbb{R}^{2d})}\,,\quad \lambda \in \Lambda\,.
\end{split}
\]
On the other hand, it is easy to check that for the Rihaczek distribution one gets
\[
R\big(\pi(-\lambda)^*\widetilde{g}_m, \pi(-\lambda)^*g_m\big)(z)=R\big(\pi(\lambda)\widetilde{g}_m, \pi(\lambda)g_m\big)(z)\,, \quad z\in \mathbb{R}^{2d}\,.
\]
Hence, for each $\lambda \in \Lambda$ we obtain
\[
\big\langle \alpha_{-\lambda}(T) g_m, \widetilde{g}_m \big\rangle_{L^2(\mathbb{R}^d)}=\big\langle \sigma_{_T}, R\big(\pi(\lambda)\widetilde{g}_m, \pi(\lambda)g_m\big) \big\rangle_{L^2(\mathbb{R}^{2d})}=\big\langle T\pi(\lambda)g_m, \pi(\lambda)\widetilde{g}_m\big\rangle_{L^2(\mathbb{R}^d)}\,.
\]
For the second equality we get
\[
\begin{split}
\big\langle T, \alpha_\lambda(\widetilde{g}_m\otimes g_m) \big\rangle_{_{\mathcal{H}\mathcal{S}}}&=\big\langle T, \pi(\lambda)\widetilde{g}_m \otimes \pi(\lambda)g_m  \big\rangle_{_{\mathcal{H}\mathcal{S}}}=\big\langle \sigma_{T}, \sigma_{\pi(\lambda)\widetilde{g}_m\otimes \pi(\lambda)g_m} \big\rangle_{L^2(\mathbb{R}^{2d})}
\\
&=\big\langle \sigma_{T}, R(\pi(\lambda)\widetilde{g}_m, \pi(\lambda)g_m) \big\rangle_{L^2(\mathbb{R}^{2d})}=\big\langle T\pi(\lambda)g_m, \pi(\lambda)\widetilde{g}_m \big\rangle_{L^2(\mathbb{R}^{d})}\,.
\end{split}
\]
We have used that the Kohn-Nirenberg symbol of $\pi(\lambda)\widetilde{g}_m \otimes \pi(\lambda)g_m$ coincides with the Rihaczek distribution of the pair of functions $\pi(\lambda)\widetilde{g}_m$ and $\pi(\lambda)g_m$ in $L^2(\mathbb{R}^d)$.

\medskip

In particular we have proved that 
\[
a_{m,n}(\lambda)=\big\langle \alpha_{-\lambda}(S_n) g_m, \widetilde{g}_m \big\rangle_{L^2(\mathbb{R}^d)}=\big\langle S_n, \alpha_{\lambda}(\widetilde{g}_m\otimes g_m)\big\rangle_{_{\mathcal{H}\mathcal{S}}}=\big\langle  \alpha_{-\lambda}(S_n), \widetilde{g}_m\otimes g_m\big\rangle_{_{\mathcal{H}\mathcal{S}}}\,,\quad \lambda \in \Lambda\,.
\]
Since $\{\alpha_\lambda (S_n)\}_{\lambda \in \Lambda;\, n=1,2,\dots,N}$ is a Riesz sequence for $\mathcal{H}\mathcal{S}(\mathbb{R}^d)$, it is in particular a Bessel sequence in $\mathcal{H}\mathcal{S}(\mathbb{R}^d)$. Hence, the sequences $\big\{\big\langle  \alpha_{-\lambda}(S_n), \widetilde{g}_m\otimes g_m\big\rangle_{_{\mathcal{H}\mathcal{S}}}\big\}_{\lambda \in \Lambda}$ belongs to $\ell^2(\Lambda)$ for $m=1, 2, \dots, M$ and $n=1, 2, \dots, N$.
\end{proof}
Once we have that $a_{m,n} \in \ell^2(\Lambda)$ for each $m=1, 2, \dots, M$ and $n=1, 2, \dots, N$, and denoting $A=[a_{m,n}]$ the corresponding $M\times N$ matrix with entries in $\ell^2(\Lambda)$, the sampling process in \eqref{dchsamp} is described by means of the discrete convolution system
\[
T=\sum_{n=1}^N\sum_{\lambda \in \Lambda} c_n(\lambda)\, \alpha_\lambda(S_n) \in V_{\bf S}^2 \longmapsto \mathbf{s}_{_T}(\lambda)=\big(A\ast_\Lambda \mathbf{c}\big)(\lambda)=\sum_{\mu\in \Lambda} A(\lambda-\mu)\, \mathbf{c}(\mu),\quad \lambda \in \Lambda\,,
\]
where $\mathbf{c}=(c_1, c_2, \dots, c_N)^\top \in \ell^2_{_N}(\Lambda):=\ell^2(\Lambda)\times \dots \times \ell^2(\Lambda)$ ($N$ times). Note that the $m$-th entry of \,$A \ast_\Lambda \mathbf{c}$\, is\, $\sum_{n=1}^N (a_{m,n}\ast_\Lambda c_{n})$. 

\medskip

First of all, the mapping $\mathcal{A}:  \ell^2_{_N}(\Lambda) \rightarrow  \ell^2_{_M}(\Lambda)$ which maps $\mathbf{c} \mapsto A \ast_\Lambda \mathbf{c}$ is a well-defined bounded operator if and only if the $M\times N$  matrix-valued function $\widehat{A}(\xi):=\big[\mathcal{F}_s^\Lambda(a_{m,n})(\xi)\big]$, a.e. $\xi \in \widehat{\Lambda}$, has entries in $L^\infty(\widehat{\Lambda})$. The needed results on discrete convolution systems $\mathcal{A}:  \ell^2_{_N}(\Lambda) \rightarrow  \ell^2_{_M}(\Lambda)$, and their relationship with frames of translates in $\ell^2_{_N}(\Lambda)$ can be found in Ref.~\cite{garcia:20}. Notice that the $m$-th component of  $A \ast_\Lambda \mathbf{c}$ is
\[
[A\ast \mathbf{c}]_{m}(\lambda)=\sum_{n=1}^N (a_{m,n} \ast_{_\Lambda} c_n)(\lambda) =\big\langle \mathbf{c}, T_{\lambda}\,\mathbf{a}^*_{m} \big\rangle_{\ell^2_{_N}(\Lambda)}\,, \quad \lambda \in \Lambda\,,
\]
where $a^*_{m,n}$ denotes the involution $a^*_{m,n}(\lambda):=\overline{a_{m,n}(-\lambda)}$, $\lambda\in \Lambda$. As a consequence, the operator 
$\mathcal{A}$ is the analysis operator of the sequence  $\big\{T_{\lambda} \mathbf{a}^*_{m}\big\}_{\lambda \in \Lambda;\, m=1,2,\dots,M}$ in $\ell^2_{_N}(\Lambda)$. Since the sequence $\big\{T_{\lambda}\,\mathbf{a}^*_{m}\big\}_{\lambda \in \Lambda;\, m=1,2,\dots,M}$  is a frame for $\ell^2_{_N}(\Lambda)$ if and only if its bounded analysis operator is injective with a closed range (see Ref.~\cite{ole:16}), it will be a frame for $\ell^2_{N}(\Lambda)$ if and only if 
\begin{equation}
\label{frame}
0<\alpha_A:=\einf_{\xi\in \widehat{\Lambda}} \lambda_{\text{min}}[\,\widehat{A}(\xi)^*\widehat{A}(\xi)] \le \beta_A:= \esup_{\xi\in \widehat{\Lambda}} \lambda_{\text{max}}[\widehat{A}(\xi)^*\widehat{A}(\xi)]<+\infty\,,
\end{equation}
where $\lambda_{\text{min}}$ (respectively, $\lambda_{\text{max}}$) denotes the smallest (respectively, the largest) eigenvalue of the positive semidefinite matrix $\widehat{A}(\xi)^*\widehat{A}(\xi)$ (see Ref.~\cite{garcia:20}).

\medskip

Concerning the duals of $\big\{T_{\lambda}\,\mathbf{a}^*_{m}\big\}_{\lambda \in \Lambda;\, m=1,2,\dots,M}$ having its same structure, consider two matrices $\widehat{A}\in \mathcal{M}_{_{M\times N}}(L^\infty(\widehat{\Lambda}))$ and $\widehat{B}\in \mathcal{M}_{_{N\times M}}(L^\infty(\widehat{\Lambda}))$, and let $\mathbf{b}_{m}$ denote the $m$-th column of the matrix $B$ associated to $\widehat{B}$. Then, the sequences $\big\{T_\lambda\,\mathbf{a}^*_{m}\big\}_{\lambda \in \Lambda;\, m=1,2,\dots,M}$ and $\big\{T_\lambda\,\mathbf{b}_{m}\big\}_{\lambda \in \Lambda;\, m=1,2,\dots,M}$ form a pair of dual frames for $\ell^2_{_N}(\Lambda)$ if and only if $\widehat{B}(\xi)\,\widehat{A}(\xi)=\mathbb{I}_{_N}$,\,\, a.e. $\xi \in \widehat{\Lambda}$; equivalently, if and only if $\mathcal{B}\,\mathcal{A}=\mathcal{I}_{\ell^2_{N}(\Lambda)}$, i.e., the convolution system 
$\mathcal{B}$ with matrix $B$ is a left-inverse of the convolution system $\mathcal{A}$ with matrix $A$. Thus,  we have the frame expansion
\[
\mathbf{c}=\sum_{m=1}^M \sum_{\lambda\in \Lambda} \big\langle \mathbf{c}, T_\lambda \mathbf{a}^*_{m} \big \rangle_{\ell^2_{_N}(\Lambda)}\, T_\lambda \mathbf{b}_{m}\quad \text{for each $\mathbf{c} \in \ell^2_{_N}(\Lambda)$}\,.
\]
Observe that a possible left-inverse $\widehat{B}(\xi)$ of the matrix $\widehat{A}(\xi)$ is given by its Moore-Penrose pseudo-inverse $\widehat{A}(\xi)^\dag=\big[\widehat{A}(\xi)^*\widehat{A}(\xi)\big]^{-1}\widehat{A}(\xi)^*$, a.e. $\xi \in \widehat{\Lambda}$.
\subsection{The sampling results}
\label{section3.4}
 Next we prove the main sampling result in this paper:
\begin{teo}
\label{tsamp1}
Suppose that  for each $T \in V_{\bf S}^2$ we consider the samples defined by \eqref{dchsamp}, and such that the matrix $A=[a_{m,n}]$, where
$a_{m,n}(\lambda)=\big\langle \alpha_{-\lambda}(S_n) g_m, \widetilde{g}_m \big\rangle_{L^2(\mathbb{R}^d)}$, \,$\lambda \in \Lambda$, satisfies conditions in Eq.~\eqref{frame}. Then,  there exist $M\geq N$ elements $H_m\in V_{\bf S}^2$, $m=1,2,\dots,M$, such that the sampling formula 
\begin{equation}
\label{fsamp1}
T=\sum_{m=1}^M\sum_{\lambda \in \Lambda} s_{_{T,m}}(\lambda) \, \alpha_\lambda(H_m)\quad \text{in $\mathcal{H}\mathcal{S}$-norm}
\end{equation}
holds for each $T\in V_{\bf S}^2$ where $\{\alpha_\lambda (H_m)\}_{\lambda \in \Lambda;\, m=1,2,\dots,M}$ is a frame for $V_{\bf S}^2$. The convergence of the series is unconditional in Hilbert-Schmidt norm. 

Moreover,  the $\ell^2$-norm of the samples $\|{\bf s}_{_T}\|_{\ell^2_{_M}}$ defines an equivalent norm to $\|T\|_{_{\mathcal{H}\mathcal{S}}}$  in $V_{\bf S}^2$, and for each $f\in L^2(\mathbb{R}^d)$ we have the pointwise expansion
\[
Tf=\sum_{m=1}^M\sum_{\lambda \in \Lambda} s_{_{T,m}}(\lambda) \, \alpha_\lambda(H_m)f\quad \text{in 
$L^2(\mathbb{R}^d)$}\,.
\]
\end{teo}
\begin{proof}
Under the hypotheses of the theorem the sequence $\big\{T_\lambda\,\mathbf{a}^*_{m}\big\}_{\lambda \in \Lambda;\, m=1,2,\dots,M}$ is a frame for $\ell^2_{_N}(\Lambda)$, and we can consider a dual frame  $\big\{T_\lambda\,\mathbf{b}_{m}\big\}_{\lambda \in \Lambda;\, m=1,2,\dots,M}$ with the same structure. As a consequence, for each 
$T=\sum_{n=1}^N\sum_{\lambda \in \Lambda} c_n(\lambda)\, \alpha_\lambda(S_n)$ in $V_{\bf S}^2$ we have
\begin{equation}
\label{frameexpansionc}
\mathbf{c}=\sum_{m=1}^M \sum_{\lambda\in \Lambda} \big\langle \mathbf{c}, T_\lambda \mathbf{a}^*_{m} \big \rangle_{\ell^2_{_N}(\Lambda)}\, T_\lambda \mathbf{b}_{m}=\sum_{m=1}^M \sum_{\lambda\in \Lambda}s_{_{T,m}}(\lambda)\, T_\lambda \mathbf{b}_{m} \quad \text{in $\ell^2_{_N}(\Lambda)$}\,,
\end{equation}
where $\mathbf{c}=(c_1, c_2, \dots, c_N)^\top \in \ell^2_{_N}(\Lambda)$. Notice that the fact that  $\big\{T_\lambda\,\mathbf{a}^*_{m}\big\}_{\lambda \in \Lambda;\, m=1,2,\dots,M}$ is a frame for $\ell^2_{_N}(\Lambda)$  and the isomorphism  $\mathcal{T}_{\bf S}$ in Eq.~\eqref{isos} give the equivalence of the norms.

The isomorphism $\mathcal{T}_{\bf S}$ defined by Eq.~\eqref{isos} applied in Eq.~\eqref{frameexpansionc} gives the sampling expansion \eqref{fsamp1},
where $H_m=K_{h_m} \in V_{\bf S}^2$ with Kohn-Nirenberg symbol $h_m=\mathcal{T}_{\sigma_{\bf S}}(\mathbf{b}_m) \in V_{\sigma_{\bf S}}^2$, $m=1,2,\dots,M$. Furthermore, since $\{\alpha_\lambda (H_m)\}_{\lambda \in \Lambda;\, m=1,2,\dots,M}$ is a frame for $V_{\bf S}^2$
the convergence of the series in the Hilbert-Schmidt norm is unconditional. Notice that 
$\mathcal{T}_{\sigma_S}(T_\lambda \mathbf{b}_m)=T_\lambda(\mathcal{T}_{\sigma_S}\mathbf{b}_m)=T_\lambda(h_m)$, where the same symbol $T_\lambda$ denotes both the translation by $\lambda$ in $\ell^2_{_N}(\Lambda)$ and in $L^2(\mathbb{R}^{2d})$ respectively. 
Notice that if $\mathbf{b}_m=\big(b_{1,m}(\lambda), b_{2,m}(\lambda), \dots, b_{N,m}(\lambda)\big)^\top$, then 
\[
H_m=\sum_{n=1}^N \sum_{\lambda \in \Lambda} b_{n,m}(\lambda)\, \alpha_\lambda(S_n)\,, \quad m=1,2,\dots,M\,.
\]
Since convergence in $\mathcal{H}\mathcal{S}$-norm implies convergence in operator norm we deduce the pointwise expansion for each $f\in L^2(\mathbb{R}^d)$.
\end{proof}

Observe that, due to conditions \eqref{frame} in Theorem \ref{tsamp1} we have necessarily $M\geq N$. Whenever $M>N$, there are infinite dual frames $\big\{T_\lambda\,\mathbf{b}_{m}\big\}_{\lambda \in \Lambda;\, m=1,2,\dots,M}$ of  
$\big\{T_\lambda\,\mathbf{a}^*_{m}\big\}_{\lambda \in \Lambda;\, m=1,2,\dots,M}$ given by the samples \eqref{samples2}. They are obtained  from the left-inverses $\widehat{B}(\xi)$ of $\widehat{A}(\xi)$ which are deduced, from the Moore-Penrose pseudo-inverse 
$\widehat{A}(\xi)^\dag$, as the $N\times M$ matrices 
\[
\widehat{B}(\xi):=\widehat{A}(\xi)^\dag+C(\xi)\big[\mathbb{I}_{_M}-\widehat{A}(\xi)\widehat{A}(\xi)^\dag\big]\,, \text{\, \,a.e. $\xi \in \widehat{\Lambda}$}\,, 
\]
where $C$ denotes any $N\times M$ matrix with entries in $L^\infty(\widehat{\Lambda})$.

More can be said in case $M=N$:
\begin{cor}
In case $M=N$, assume that the conditions
\begin{equation}
\label{riesz}
0<\einf_{\xi\in \widehat{\Lambda}} \big|\det [\widehat{A}(\xi)]\big| \le \esup_{\xi\in \widehat{\Lambda}} \big|\det [\widehat{A}(\xi)]\big|<+\infty
\end{equation}
hold. Then, there exist $N$ unique elements $H_n$, $n=1,2,\dots, N$, in $V_{\bf S}^2$ such that the associated sequence
$\big\{\alpha_\lambda(H_{n})\big\}_{\lambda\in \Lambda;\,n=1,2,\dots, N}$ is a Riesz basis for $V_{\bf S}^2$ and the sampling formula
\[
T=\sum_{n=1}^N\sum_{\lambda \in \Lambda} s_{_{T,n}}(\lambda)\, \alpha_\lambda(H_n)\quad \text{in $\mathcal{H}\mathcal{S}$-norm}
\]      
holds for each $T\in V_{\bf S}^2$.
Moreover, the interpolation property $\big\langle \alpha_{-\lambda}(H_m) g_n, \widetilde{g}_n \big\rangle=\delta_{m,n}\delta_{\lambda,0}$, where $\lambda \in \Lambda $ and $m,n=1,2,\dots,N$, holds.
\end{cor}
\begin{proof}
In this case, the square matrix $\widehat{A}(\xi)$ is invertible and the statement \eqref{riesz} in corollary  is equivalent to condition $0<\alpha_A \le \beta_A <+\infty$ in \eqref{frame}; besides,  
any Riesz basis has a unique dual basis. The uniqueness of the coefficients in a Riesz basis expansion gives the interpolation property.
\end{proof}

In particular, for the case $N=M=1$ we have:
\begin{cor}
Assume that the sequence $\mathbf{a}=\{a(\lambda)\}_{\lambda \in \Lambda}$, where $a(\lambda)=\big\langle \alpha_{-\lambda}(S) g, \widetilde{g}\, \big\rangle_{L^2(\mathbb{R}^d)}$, \,$\lambda \in \Lambda$, for a fixed pair of functions $g, \widetilde{g}\in L^2(\mathbb{R}^d)$,  satisfies the conditions
\begin{equation}
\label{baseriesz}
0<\einf_{\xi \in \widehat{\Lambda}} |\mathcal{F}_s^\Lambda (\mathbf{a})(\xi)| \le \esup_{\xi \in \widehat{\Lambda}} |\mathcal{F}_s^\Lambda(\mathbf{a})(\xi)| <\infty\,.
\end{equation}
Then, there exists a unique $H\in V^2_{_S}$ such that the sequence $\big\{\alpha_\lambda(H)\big\}_{\lambda\in \Lambda}$ is a Riesz basis for $V_{_S}^2$ and the sampling formula
\[
T=\sum_{\lambda \in \Lambda} \big\langle \alpha_{-\lambda}(T) g, \widetilde{g}\, \big\rangle_{L^2(\mathbb{R}^d)}\, \alpha_\lambda(H)\quad \text{in $\mathcal{H}\mathcal{S}$-norm}
\]
holds for each $T\in V_{_S}^2$.
Moreover, the interpolation property $\big\langle \alpha_{-\lambda}(H) g, \widetilde{g}\, \big\rangle=\delta_{\lambda,0}$, $\lambda \in \Lambda $, holds; in particular, 
$\big\langle Hg, \widetilde{g}\, \big\rangle=1$.
\end{cor}
It is worth to remark that  in the above sampling result is not necessary that the operators in $V_{_S}^2$ have a bandlimited Kohn-Nirenberg symbol as in 
Ref.~\cite[Theorem 2]{grochenig:14}.

\medskip

The bandlimited case is obtained as a particular case. Let $\Lambda=a\mathbb{Z}^d \times b\mathbb{Z}^d$ be a lattice in $\mathbb{R}^{2d}$ with $a,b>0$. Assume that the generator $S$ of $V^2_{_S}$ is a bandlimited operator to $Q:=[\frac{-1}{2a}, \frac{1}{2a}]^d\times [\frac{-1}{2b}, \frac{1}{2b}]^d$, i.e., it belongs to $OPW^2(Q):=\{T\in \mathcal{H}\mathcal{S}(\mathbb{R}^d) \,:\, \text{supp\,} \widehat{\sigma}_{_T} \subseteq Q\}$. Then any $T\in V^2_{_S}$ also belongs to $OPW^2(Q)$. In case conditions \eqref{baseriesz} are satisfied, any $T\in V^2_{_S}$ can be recovered from its diagonal channel samples as
\[
T=\sum_{\lambda \in \Lambda} \big\langle T\pi(\lambda)g, \pi(\lambda)\widetilde{g}\,\big\rangle_{L^2(\mathbb{R}^d)}\, \alpha_\lambda(H)\quad \text{in $\mathcal{H}\mathcal{S}$-norm}\,,
\]
where $H=\sum_{\lambda \in \Lambda} b(\lambda)\,\alpha_\lambda(S)$ in $V^2_{_S}$ is obtained from the sequence $\mathbf{b}=\{b(\lambda)\}_{\lambda \in \Lambda}$ in $\ell^2(\Lambda)$ such that $\mathcal{F}_s^\Lambda (\mathbf{b})(\xi)\,\mathcal{F}_s^\Lambda (\mathbf{a})(\xi)=1$, a.e. $\xi \in \widehat{\Lambda}$.

\medskip

In Ref.~\cite{grochenig:14}  the reconstruction of pseudodifferential operators with a bandlimited Kohn-Nirenberg symbol is considered. In particular, Theorem 2 of the same reference proves that, under some appropriate assumptions, for any $T\in OPW^2(Q)$ we have
\[
\sigma_{_T}=\frac{1}{(ab)^d} \sum_{\lambda \in \Lambda} \big\langle T\pi(\lambda)g, \pi(\lambda)g\,\big\rangle_{L^2(\mathbb{R}^d)}\, T_\lambda \big({\rm sinc}_{a,b}\ast k \big)\quad \text{in  $L^2(\mathbb{R}^{2d})$}\,,
\]
where the function $k$, independent of $T$, belongs to $L^1(\mathbb{R}^{2d})$ and $\rm{sinc}_{a,b}$ denotes the {\em sinc function} adapted to the lattice 
$\Lambda=a\mathbb{Z}^d \times b\mathbb{Z}^d$, namely 
\[
{\rm sinc}_{a,b}(x)=\prod_{j=1}^d \dfrac{\sin \pi a x_j}{\pi a x_j} \prod_{j=d+1}^{2d} \dfrac{\sin \pi b x_j}{\pi b x_j}\,, \quad x\in \mathbb{R}^{2d}\,.
\]
Using the Kohn-Nirenberg transfom, the above sampling formula for $\sigma_{_T}$ can be written as 
\[
T=\frac{1}{(ab)^d} \sum_{\lambda \in \Lambda} \big\langle T\pi(\lambda)g, \pi(\lambda)g\,\big\rangle_{L^2(\mathbb{R}^d)}\, \alpha_\lambda (k\ast K_{\rm{sinc}_{a,b}} )\quad \text{in $\mathcal{H}\mathcal{S}$-norm}\,,
\]
where $k\ast K_{{\rm sinc}_{a,b}}$ denotes the Hilbert-Schmidt operator obtained from the convolution of the function $k$ and the operator $K_{{\rm sinc}_{a,b}}$; we have also used the following result:

\begin{lema}
Let $K_f$ be an operator in $\mathcal{H}\mathcal{S}(\mathbb{R}^d)$ with Kohn-Nirenberg symbol $f\in L^2(\mathbb{R}^{2d})$, and let $g$ a function in $L^1(\mathbb{R}^{2d})$. Then we have that $K_{g\ast f}=g\ast K_f$.
\end{lema}
\begin{proof}
Recall that the convolution $g\ast K_f$ is the operator in $\mathcal{H}\mathcal{S}(\mathbb{R}^d)$ defined by the operator-valued integral (in weak sense)
\[
g\ast K_f=\int_{\mathbb{R}^{2d}} g(z)\, \alpha_z(K_f) dz\,,
\]
i.e.,
\[
\Big\langle \big(\int_{\mathbb{R}^{2d}} g(z)\, \alpha_z(K_f) dz\big)\varphi, \psi\Big\rangle_{L^2(\mathbb{R}^d)}=\int_{\mathbb{R}^{2d}} g(z)\, \langle \alpha_z(K_f)\varphi, \psi \rangle dz\,,\quad \varphi, \psi \in L^2(\mathbb{R}^d)\,.
\]
See the details in Refs.~\cite{luef:18,skret:20,werner:84}. Since the map $\mathcal{K}: \mathcal{H}\mathcal{S}(\mathbb{R}^d) \rightarrow L^2(\mathbb{R}^{2d})$ such that $\mathcal{K}(K_f)=f $ is a unitary operator and bounded operators conmute with convergent integrals \cite[Proposition 2.4]{luef:18} we get
\[
\mathcal{K}(g\ast K_f)=\int_{\mathbb{R}^{2d}} g(z)\,\mathcal{K}\big(\alpha_z(K_f)\big)dz=\int_{\mathbb{R}^{2d}} g(z)\,\mathcal{K}\big(K_{T_zf}\big)dz=
\int_{\mathbb{R}^{2d}} g(z)\,f(\cdot-z)dz=g\ast f\,,
\]
that is, $K_{g\ast f}=g\ast K_f$.
\end{proof}

\bigskip

In the same manner we can consider {\em average sampling} in $V_{\bf S}^2$. Namely, for any $T\in V_{\bf S}^2$, its {\em average samples} at $\Lambda$ are defined by
\[
\big\langle T, \alpha_\lambda(Q_m)\big \rangle_{_{\mathcal{H}\mathcal{S}}}\,,\quad \lambda \in \Lambda\,,\,\,\, m=1, 2, \dots,M\,,
\]
from $M$ fixed operators $Q_1, Q_2, \dots, Q_M$ in $\mathcal{H}\mathcal{S}(\mathbb{R}^d)$, not necessarily in $V_{\bf S}^2$. Observe that, having in mind Eq.~\eqref{eledos} in Lemma \ref{lemmasamples}, the diagonal channel samples defined in Eq.~\eqref{dchsamp} are a particular case of average sampling where $Q_m=\widetilde{g}_m\otimes g_m$, $m=1, 2, \dots, M$. The average samples of any $T=\sum_{n=1}^N\sum_{\mu \in \Lambda} c_n(\mu)\, \alpha_\mu(S_n)$ can be also expressed as a discrete convolution system in $\ell^2_{_N}(\Lambda)$. Indeed, for $m=1, 2, \dots, M$ we have
\[
\begin{split}
\big\langle T, \alpha_\lambda(Q_m) \big\rangle_{_{\mathcal{H}\mathcal{S}}} = & \big\langle \sigma_{_T}, T_\lambda \sigma_{_{Q_m}} \big\rangle_{_{L^2(\mathbb{R}^{2d})}}=
\big\langle \sum_{n=1}^N\sum_{\mu \in \Lambda} c_n(\mu)\, T_{\mu}  \sigma_{_{S_n}}, T_\lambda \sigma_{_{Q_m}} \big\rangle_{_{L^2(\mathbb{R}^{2d})}} \\
=&\sum_{n=1}^N\sum_{\mu \in \Lambda} c_n(\mu) \langle T_{\mu}  \sigma_{_{S_n}}, T_\lambda \sigma_{_{Q_m}} \rangle_{_{L^2(\mathbb{R}^{2d})}} =\sum_{n=1}^N\sum_{\mu \in \Lambda} c_n(\mu) \langle \sigma_{_{S_n}}, T_{\lambda-\mu} \sigma_{_{Q_m}} \rangle_{_{L^2(\mathbb{R}^{2d})}} \\
=&\sum_{n=1}^N \big(a_{m,n} \ast_\Lambda c_n\big)(\lambda)\,,\quad \lambda \in \Lambda\,,
\end{split}
\]
where $a_{m,n}(\mu):=\big\langle \sigma_{_{S_n}}, T_\mu \sigma_{_{Q_m}} \big\rangle_{_{L^2(\mathbb{R}^{2d})}}=\big\langle S_n, \alpha_\mu(Q_m)\big \rangle_{_{\mathcal{H}\mathcal{S}}}$, \,$\mu \in \Lambda$, and 
$\sigma_{_{S_n}}$, $\sigma_{_{Q_m}}$ are the Kohn-Nirenberg symbols of $S_n$, $Q_m$ respectively. 

Observe that, for each $m=1, 2, \dots, M$ and $n=1, 2, \dots, N$, the sequence $\{a_{m,n}(\lambda)\}_{\lambda \in \Lambda}$ belongs to $\ell^2(\Lambda)$ since, in particular, $\{T_\lambda \sigma_{_{S_n}}\}_{\lambda \in \Lambda;\, n=1,2,\dots,N}$ is a Bessel sequence in $L^2(\mathbb{R}^{2d})$. 

\begin{cor}
Assume that the matrix $A=[a_{m,n}]$ with entries $a_{m,n}(\lambda)=\big\langle S_n, \alpha_\lambda(Q_m)\big \rangle_{_{\mathcal{H}\mathcal{S}}}$, \,$\lambda \in \Lambda$, satisfies conditions in \eqref{frame}. Then,  there exist $M\geq N$ operators $H_m\in V_{\bf S}^2$, $m=1,2,\dots,M$, such that the sampling formula 
\[
T=\sum_{m=1}^M\sum_{\lambda \in \Lambda} \big\langle T, \alpha_\lambda(Q_m)\big \rangle_{_{\mathcal{H}\mathcal{S}}} \, \alpha_\lambda(H_m)\quad \text{in $\mathcal{H}\mathcal{S}$-norm}
\]
holds for each $T\in V_{\bf S}^2$ where $\{\alpha_\lambda (H_m)\}_{\lambda \in \Lambda;\, m=1,2,\dots,M}$ is a frame for $V_{\bf S}^2$. The convergence of the series is unconditional in Hilbert-Schmidt norm.
\end{cor}
The above sampling formula was obtained in Ref.~\cite{garcia:21} by using the Weyl symbols of $S_n$ and $Q_m$ instead of their Kohn-Nirenberg symbols.
Finally, it is worth to mention that each sampling result in this section admit a kind of converse result; see the details in Theorems 1-2 and Corollary 3 of Ref.~\cite{garcia:21}.

\subsubsection*{An illustrative example}
Assume that in $V_{\bf S}^2$ we have $N$ stable generators of the form $S_n=\varphi_n \otimes \widetilde{\varphi}_n$ with $\varphi_n, \widetilde{\varphi}_n \in \mathcal{S}_0(\mathbb{R}^d)$, $n=1, 2, \dots,N$. In this regard, note that in order to apply Theorem \ref{Ngenerators} we have that $\mathcal{F}_W(\varphi_n \otimes \widetilde{\varphi}_n)(z)={\rm e}^{\pi i x\cdot \omega}\,V_{\widetilde{\varphi}_n}\varphi_n(z)$, $z=(x, \omega)\in \mathbb{R}^{2d}$ (see Ref.~\cite{skret:20}). 

\medskip

Next, for each $T\in V_{\bf S}^2$ we consider the diagonal channel samples $\big\langle T\pi(\lambda)g_m, \pi(\lambda)\widetilde{g}_m \big\rangle_{L^2(\mathbb{R}^{d})}$, $\lambda \in \Lambda$ and $m=1, 2 ,\dots M$, with $g_m, \widetilde{g}_m \in \mathcal{S}_0(\mathbb{R}^d)$. In this case, for $m=1, 2, \dots,M$ and $n=1, 2, \dots,N$, we get
\[
\begin{split}
a_{m,n}(\lambda)&=\big\langle \alpha_{-\lambda}(\varphi_n \otimes \widetilde{\varphi}_n) g_m, \widetilde{g}_m \big\rangle_{L^2(\mathbb{R}^d)}=
\big\langle \big(\varphi_n \otimes \widetilde{\varphi}_n\big)\pi(\lambda)g_m, \pi(\lambda)\widetilde{g}_m\big\rangle_{L^2(\mathbb{R}^d)}\\
&=\big\langle \langle \pi(\lambda)g_m, \widetilde{\varphi}_n\rangle \varphi_n, \pi(\lambda)\widetilde{g}_m\big\rangle_{L^2(\mathbb{R}^d)}=
\overline{V_{g_m}\widetilde{\varphi}_n(\lambda)}\, V_{\widetilde{g}_m}\varphi_n(\lambda)\,, \quad \lambda \in \Lambda\,.
\end{split}
\]
From Proposition 4.1 in Ref.~\cite{skret:20} we deduce that the sequences $\big\{a_{m,n}(\lambda)\big\}_{\lambda \in \Lambda}$ belong to $\ell^1(\Lambda)$ and, as a consequence, the entries in the transfer matrix $\widehat{A}$ are continuous functions on the compact $\widehat{\Lambda}$. In order to apply Theorem \ref{tsamp1} conditions in Eq.~\eqref{frame} reduce to
\[
\det [\widehat{A}(\xi)^*\widehat{A}(\xi)] \neq 0\quad \text{for all $\xi \in \widehat{\Lambda}$}\,.
\]
Under the above circumstances, any  $T\in V_{\bf S}^2$, which is nothing but $T=\sum_{n=1}^N \mathcal{G}_{\mathbf{c}_n}^{\widetilde{\varphi}_n, \varphi_n}$ a finite sum of Gabor multipliers, can be recovered, in a stable way, from its diagonal channel samples
$\big\langle T\pi(\lambda)g_m, \pi(\lambda)\widetilde{g}_m \big\rangle_{L^2(\mathbb{R}^{d})}$, \,$\lambda\in \Lambda$ and $m=1,2,\dots,M$.
\subsection{Sampling in a sub-lattice of $\Lambda$}
\label{section3.5}
Let $\Lambda'$ be a sub-lattice of $\Lambda$ with finite index $L$, i.e., the quotient group $\Lambda/\Lambda'$ has finite order $L$. We consider 
$\{\lambda_1, \lambda_2, \dots, \lambda_L\}$ a set of representatives of the cosets of $\Lambda'$. That is, the lattice $\Lambda$
be decomposed as
\[
\Lambda= \bigcup_{l=1}^L(\lambda_l+\Lambda')\quad \text{with}\quad (\lambda_l+\Lambda') \cap (\lambda_{l'}+\Lambda')=\varnothing \,\,\text{for $l\neq l'$}\,.
\]
Thus, the space $V_{\bf S}^2$ can be written as
\[
\begin{split}
V_{\bf S}^2&=\Big\{\sum_{n=1}^N\sum_{\lambda \in \Lambda} c_n(\lambda)\,\alpha_\lambda (S_n)\, :\,  c_n \in \ell^2(\Lambda)\Big\}=
\Big\{\sum_{n=1}^N\sum_{l=1}^L \sum_{\mu\in \Lambda'} c_n(\lambda_l+\mu)\,\alpha_{\lambda_l+\mu}(S_n)\Big\}\\
&=\Big\{\sum_{n=1}^N\sum_{l=1}^L \sum_{\mu\in \Lambda'} c_{nl}(\mu)\, \alpha_\mu (S_{nl})\, :\,  c_{nl} \in \ell^2(\Lambda')\Big\}\,,
\end{split}
\]
where $c_{nl}(\mu):=c_n(\lambda_l+\mu)$ and $S_{nl}:=\alpha_{\lambda_l}(S_n)$, and  the new index $nl$ goes  from $11$ to $NL$. As a consequence, the subspace $V_{\bf S}^2$ can be rewritten as $V_{\widetilde{\bf S}}^2$ with $NL$ generators $\widetilde{\bf S}=\{S_{nl}\}$  and coefficients $c_{nl}$ in 
$\ell^2(\Lambda')$.

\medskip

Let $T=\sum_{n=1}^N\sum_{l=1}^L \sum_{\nu\in \Lambda'} c_{nl}(\nu)\, \alpha_\nu (S_{nl})$ be in $V_{\bf S}^2$; its samples 
$\big\langle \alpha_{-\mu}(T) g_m, \widetilde{g}_m \big\rangle_{L^2(\mathbb{R}^d)}$, $\mu \in \Lambda'$, can be expressed by
\[
\begin{split}
s_{_{T,m}}(\mu):=&\big\langle \alpha_{-\mu}(T) g_m, \widetilde{g}_m \big\rangle_{L^2(\mathbb{R}^d)}= \Big\langle \sum_{n=1}^N \sum_{l=1}^L \sum_{\nu \in \Lambda'} c_{nl}(\mu)\, \alpha_{\nu-\mu}(S_{nl})g_m, \widetilde{g}_m \Big\rangle_{L^2(\mathbb{R}^d)} \\
=&\sum_{n=1}^N \sum_{l=1}^L \sum_{\nu \in \Lambda'} c_{nl}(\nu) \big\langle \alpha_{\nu-\mu}(S_{nl})g_m, \widetilde{g}_m \big\rangle_{L^2(\mathbb{R}^d)}
=\sum_{n=1}^N \sum_{l=1}^L \big(a_{m,nl} \ast_{\Lambda'} c_{nl}\big)(\mu)\,, \quad \mu \in \Lambda'\,,
\end{split}
\]
where $a_{m,nl}(\nu):=\big\langle \alpha_{-\nu}(S_{nl}) g_m, \widetilde{g}_m \big\rangle_{L^2(\mathbb{R}^d)}$, \,$\nu \in \Lambda'$. Hence, Theorem \ref{tsamp1} gives:
\begin{cor}
Let $A=[a_{m,nl}]$ be the $M\times NL$ matrix with entries 
\[
a_{m,nl}(\nu)=\big\langle \alpha_{-\nu}(S_{nl}) g_m, \widetilde{g}_m \big\rangle_{L^2(\mathbb{R}^d)}\,, \quad \nu \in \Lambda'\,,
\]
for $m=1, 2, \dots, M$ and $nl=11, 12, \dots, NL$. Assume that $A$ satisfies conditions in \eqref{frame} with respect to the dual $\widehat{\Lambda' }$. Then,  there exist $M\geq NL$ operators $H_m\in V_{\bf S}^2$, $m=1,2,\dots,M$, such that the sampling formula 
\[
T=\sum_{m=1}^M \sum_{\mu \in \Lambda'} s_{_{T,m}}(\mu) \, \alpha_\mu(H_m)\quad \text{in $\mathcal{H}\mathcal{S}$-norm}
\]
holds for each $T\in V_{\bf S}^2$ where $\{\alpha_\mu (H_m)\}_{\mu \in \Lambda';\, m=1,2,\dots,M}$ is a frame for 
$V_{\bf S}^2$. The convergence of the series is unconditional in Hilbert-Schmidt norm.
\end{cor}

\bigskip

\noindent{\bf Acknowledgments:}
The author thanks {\em Universidad Carlos III de Madrid} for granting him a sabbatical year in 2020-21.
This work has been supported by the grant MTM2017-84098-P from the Spanish {\em Ministerio de Econom\'{\i}a y Competitividad (MINECO)}.


\end{document}